\documentclass[11pt]{article}
\usepackage{hyperref,amssymb,amsfonts,amsmath,graphicx,verbatim,eufrak,amsthm}

\usepackage{fullpage}
\usepackage{color}

\newtheorem{thm}{Theorem}
\newtheorem{thmtool}{Theorem}[section]
\newtheorem{corollary}[thmtool]{Corollary}
\newtheorem{lem}[thmtool]{Lemma}
\newtheorem{prop}[thmtool]{Proposition}
\newtheorem{clm}[thmtool]{Claim}

\newtheorem*{thm*}{Theorem}

\theoremstyle{definition}

\theoremstyle{remark}

\numberwithin{equation}{section}

\newcommand{\N}{\mathbb N}

\newcommand{\FAS}{\mathrm {FAS}}
\newcommand{\danger}{\mathsf {dang}}
\newcommand{\avdanan}{\overline{\danger}}
\newcommand{\ed}{e_{\text{double}}}

\newcommand{\degb}{\deg_{B}}
\newcommand{\degm}{\deg_{M}}

\newcommand{\eps}{\varepsilon}

\def\squareforqed{\hbox{\rlap{$\sqcap$}$\sqcup$}}
\def\qed{\ifmmode\squareforqed\else{\unskip\nobreak\hfil
\penalty50\hskip1em\null\nobreak\hfil\squareforqed
\parfillskip=0pt\finalhyphendemerits=0\endgraf}\fi}

\begin{document}
\title{Biased orientation games}
\author{Ido Ben-Eliezer\thanks{School of Computer Science, Raymond and Beverly Sackler Faculy of Exact Sciences, Tel Aviv University, Tel Aviv 69978, Israel, e-mail: \textbf{idobene@post.tau.ac.il}. Research supported in part by an ERC advanced grant.} \and Michael Krivelevich \thanks{School of Mathematical Sciences, Raymond and Beverly Sackler Faculty of Exact Sciences, Tel Aviv University,
Tel Aviv 69978, Israel, e-mail: \textbf{ krivelev@post.tau.ac.il}.
Research supported in part by USA-Israel BSF grant 2006322, and by
grant 1063/08 from the Israel Science Foundation.} \and Benny
Sudakov \thanks{Department of Mathematics, UCLA, Los Angeles, CA
90095. Email: \textbf{bsudakov@math.ucla.edu}. Research supported in
part by NSF CAREER award DMS-0812005 and by a USA-Israel BSF
grant.}} \maketitle
\begin{abstract}
We study biased {\em orientation games}, in which the board is the
complete graph $K_n$, and Maker and Breaker take turns in directing
previously undirected edges of $K_n$. At the end of the game, the
obtained graph is a tournament. Maker wins if the tournament has
some property $\mathcal P$ and Breaker wins otherwise.

We provide bounds on the bias that is required for a Maker's win and
for a Breaker's win in three different games. In the first game
Maker wins if the obtained tournament has a cycle. The second game
is Hamiltonicity, where Maker wins if the obtained tournament
contains a Hamilton cycle. Finally, we consider the $H$-creation
game, where Maker wins if the obtained tournament has a copy of some
fixed graph $H$.
\end{abstract}
\section{Introduction}

In this work we study {\em orientation games}. The board consists of
the edges of the complete graph $K_n$. In the $(p:q)$ game the two
players, called {\em Maker} and {\em Breaker}, take turns in
orienting (or directing) previously undirected edges. Maker starts
the game and at each round, Maker directs at most $p$ edges and then
Breaker directs at most $q$ edges (usually, we consider the case
where $p=1$ and $q$ is large). The game ends where all the edges are
oriented, and then we obtain a {\em tournament}. Maker then wins if
the tournament has some fixed property $\mathcal P$, and Breaker
wins otherwise. Here we focus on the $1:b$ game, which is referred
to as the $b$-biased game. We stress that at each round, each player
has to orient at least one edge, so the number of rounds is clearly
bounded. Also, Maker (respectively, Breaker) can orient up to $p$
edges (respectively, up to $q$ edges) and hence by bias monotonicity
every property $\mathcal P$ admits some threshold $t(n,\mathcal P)$
so that Maker wins the $b$-biased game if $b<t(n,\mathcal P)$ and
Breaker wins the $b$-biased game if $b \geq t(n,\mathcal P)$

This game is a variant of the well studied classical Maker-Breaker
game, which is defined by a hypergraph $(X,\mathcal F)$ and bias
$(p:q)$. In that game, at each round Maker claims $p$ elements of
$X$, and Breaker claims $q$ elements of $X$. Maker wins if by the
end of the game he claimed all the elements of some hyperedge $A \in
\mathcal F$, and Breaker wins otherwise. Usually, a typical problem
goes as follows. Given a game hypergraph $H = (X, \mathcal F)$,
$|X|=n$, determine or estimate the threshold function $t_H$ such
that if $b > t_H$ then Maker wins in a $(1:t_H)$ game, and if $b
\leq t_H$ then Breaker wins in a $(1:t_H)$ game. There has been a
long line of research that studies the bias threshold of various
games (see, e.g.,~\cite{Beck08,BL00,CE78,GS09,Krivelevich11,KS08}
and their references).

Here we study orientation games for the following three properties.

\paragraph{Creating a cycle.} Maker wins if the obtained tournament
contains a cycle, and Breaker wins otherwise. It is well known that
a tournament contains a cycle if and only if it contains a directed
triangle (cycle of length $3$). This is a relatively old question
which has already been studied by Alon (unpublished result) and by
Bollob\'{a}s and Szab\'{o}~\cite{BS98}, and here we improve their
results.

\paragraph{Creating a Hamilton cycle.} Here Maker wins if the
obtained tournament contains a Hamilton cycle, and Breaker wins
otherwise. Recently, the second author~\cite{Krivelevich11} solved a
long standing question and provided tight bounds on the bias
threshold for Maker win in the classical Maker-Breaker Hamiltonicity
game. We use a variant of his approach, together with a new
application of the Gebauer-Szab\'{o} method~\cite{GS09} and give
tight bounds in our case as well.

\paragraph{Creating a copy of $H$.} Here we are given a fixed graph
$H$. Maker wins if the obtained tournament contains a copy of $H$,
and Breaker wins otherwise. We provide both upper and lower bounds,
and give some nearly tight bounds for specific cases. We conjecture
that the correct threshold is closely related to the size of the
{\em minimum feedback arc set} of $H$, and provide some results that
support this conjecture.

\paragraph{Our results.}

In this work we study the cycle game, the Hamiltonicty game and the
$H$-creation game. We stress that in all these games Maker wins if
the obtained tournament has the desired property, no matter who
directs each edge of a winning directed subgraph. Our first theorem
considers the cycle creating game. It is easy to observe that if $b
\geq n-2$ then Breaker has a winning strategy (for completeness, we
give a detailed proof in Section~\ref{section:cycle-game}).
Bollob\'{a}s and Szab\'{o}~\cite{BS98} proved that if $b =
(2-\sqrt{3})n$, Maker wins the game and conjectured that the correct
threshold is $b = n-2$.

In this work we provide a simple argument that improves their
result. We have the following.

\begin{thm}[The cycle game]
\label{thm:creating-cycle} For every $b \leq n/2-2$, Maker has a
strategy guaranteeing a cycle in the $b$-biased orientation game.
\end{thm}

The second game we consider is the Hamiltonicity game, where Maker
wins if and only if the obtained tournament contains a Hamilton
cycle. Here we apply old and recent
techniques~\cite{CE78,GS09,Krivelevich11} to get tight bounds on the
bias threshold for a win of Breaker.

\begin{thm} [The Hamiltonicity game]
\label{thm:create-hamilton} \mbox{}
\begin{enumerate}
\item If $b \geq \frac{n(1+o(1))}{\ln{n}}$, Breaker has a
strategy to guarantee that in the $b$-biased orientation game the
obtained tournament has a vertex of in-degree $0$, and in particular
to win the Hamiltonicity game. \label{item:hamilton-upper-bound}
\item If $b \leq \frac{n(1+o(1))}{\ln{n}}$, Maker has a
strategy guaranteeing a Hamilton cycle in the $b$-biased orientation
game. \label{item:hamilton-lower-bound}
\end{enumerate}
\end{thm}

In the $H$-creation game we have some partial results. We conjecture
that the bias that guarantees Maker's win depends on the minimum
feedback arc set of $H$, and support this result for graphs with a
small feedback arc set. We will give and discuss corresponding
notion in Section~\ref{sec:h-creation}.

\section{Preliminaries}

Let $K_n$ be the complete graph on $n$ vertices, a tournament is an
orientation of $K_n$. A directed graph is called {\em oriented} if
it contains nor loops neither cycles of length $2$. Every oriented
graph is a subgraph of a tournament. A directed graph is {\em
strongly connected} if for every two vertices $u,v$ there is a
directed path from $u$ to $v$ and a directed path from $v$ to $u$.
All directed graphs we consider here are {\em oriented}, i.e., do
not have parallel or opposite edges.

All logarithms are in base $2$ unless stated otherwise.

The {\em classical} Maker and Breaker game goes as follows. Given a
hypergraph $H=(V,\mathcal F)$, at every round Maker occupies $p$
elements from $V$, and then Breaker occupies $q$ elements from $V$.
By the end of the game, Maker wins if he occupies completely some
hyperedge in $\mathcal F$, and otherwise Breaker wins. The well
known results of Erd\H{o}s and Selfridge~\cite{ES73} and
Beck~\cite{Beck82} give a sufficient condition for a Maker's win.
\begin{thm}
\label{thm:Erdos-selfride-beck} Suppose that Maker and Breaker play
a $(p:q)$-game on a hypergraph $H=(V,\mathcal F)$. If
$$
\sum_{A \in \mathcal F} (q+1)^{-\frac{|A|}{p}} < \frac{1}{q+1},
$$
Then Breaker has a winning strategy, even if Maker starts the game.
\end{thm}

An orientation game is defined by a series of moves by Maker and
Breaker. In every {\em round}, Maker orients $1 \leq m_t \leq p$
edges (usually in our settings $p=1$) and Breaker orients $1 \leq
b_t \leq q$ edges (usually in our settings $q = \omega(1)$). The
game ends where all the edges are oriented, so the obtained graph is
a tournament. Maker wins if the tournament has some predetermined
property $\mathcal P$, otherwise Breaker wins.

We denote by $H_t$ the obtained oriented graph after $t$ rounds.
Clearly, this graph has at most $(p+q) \cdot t$ edges.

Given a directed graph $G = (V,E)$, we write $(u,v) \in E$ if there
is an edge from $u$ to $v$. Given a set $A \subseteq V$, we let
$$N^+(A) = \{u \in V \setminus A : \exists v \in A, (v,u) \in E \},$$
and
$$N^-(A) = \{u \in V \setminus A : \exists v \in A, (u,v) \in E \}.$$

A tournament $T$ on $n$ vertices is {\em transitive} if there is a
bijection $\sigma : V(T) \to [n]$ such that for every edge $(u,v)
\in E(T)$, $\sigma(u) < \sigma(v)$. A tournament $T = (V,E)$ is
$k$-colorable if there is a partition of $V$ into $k$ sets
$V_1,\ldots,V_k$ such that the induced tournament on each $V_i$ is
transitive. Thus, a transitive tournament is $1$-colorable.

\section{The cycle game}
\label{section:cycle-game}

In this section we prove Theorem~\ref{thm:creating-cycle}. Namely,
we show that in the $(n/2-2)$-biased game Maker can create a cycle.
For the sake of completeness we also prove that in the
$(n-2)$-biased game Breaker can create an acyclic tournament.

\paragraph{Breaker's strategy.} Suppose that $b \geq n-2$, we show
that Breaker can block all cycles in the graph as follows. Whenever
Maker orients an edge from $u$ to $v$, Breaker responds by orienting
all edges from $u$ to every vertex $w \in V(K_n)$ such that the edge
$uw$ has not been oriented yet. Clearly, Breaker in his turn has to
orient at most $n-2$ edges.

We proceed by proving that no cycle is created when Breaker applies
this strategy. Indeed, suppose that a cycle $C$ is created and let
$(u,v)$ the first edge in $C$ that was oriented (by either Maker or
Breaker), and suppose also that $(w,u) \in C$. If Maker orients the
edge from $u$ to $v$, by the strategy above Breaker responses by
orienting the edge from $u$ to $w$, and thus $(w,u) \notin C$. If
Breaker orients the edge from $u$ to $v$, he did it because Maker
oriented some other edge from $u$ to some vertex $z$. In this case,
again Breaker will also orient the edge from $u$ to $w$, and
therefore again $(w,u) \notin C$. We conclude that no cycle is
created.

\paragraph{Maker's strategy.} Our main lemma states that Maker has a
strategy so $H_t$ contains a directed path of length $t$ throughout
the game.
\begin{lem}
\label{lem:maker-strategy-creating-cycle} In the $b$-biased game,
Maker has a strategy $S_M$ such that for every $t \leq n-1$, the
graph $H_t$ obtained after $t$ rounds contains a directed path of
length $t$.
\end{lem}
\begin{proof}
We prove by induction that assuming that there are no cycles in the
graph, Maker can extend a longest path by one, no matter how Breaker
plays. Clearly Maker can create a path of length $1$ at the first
round. Suppose that the longest path in $E(H_t)$ is $P_t =
u_1,u_2,\ldots,u_r$, where $r \geq t$. Let $v$ be a vertex not in
the path. Let $k$ be the maximal index such that there is no edge
from $v$ to $u_k$. This is well defined as if there is an edge from
$v$ to $u_1$ then $v,u_1,\ldots,u_r$ is a longer path, contradicting
the maximality of $P_t$.

Observe first that if there is an edge in the opposite direction
from $u_k$ to $v$ then $u_1,\ldots,u_r$ is not a maximal path.
Indeed, if $k=r$ then $u_1,\ldots,u_r,v$ is a longer path; Otherwise
by the definition of $k$ there is an edge from $v$ to $u_{k+1}$ and
therefore $u_1,\ldots,u_k,v,u_{k+1},\ldots,r$ is a longer path, and
in both cases this contradicts the maximality of $P_t$.

Therefore Maker in his turn orients the edge from $u_k$ to $v$ and
creates a path of length at least $r+1$, and the result follows.
\end{proof}

The strategy of Maker is as follows. At each round, if he can close
a cycle he does so and wins. Otherwise, he increases the length of a
longest directed path. We next show that after large enough number
of rounds, Breaker cannot block all possible cycles.

\paragraph{Proof of Theorem~\ref{thm:creating-cycle}.} As long as Maker cannot orient an
edge such that a cycle is created, Maker can extend a longest
directed path by $1$ by
Lemma~\ref{lem:maker-strategy-creating-cycle}. After $t$ rounds,
there is a path $P_t$ of length at least $t$. Let $V_t = V(P_t)$.
There are ${t \choose 2} - t$ potential edges in $G[V_t]$ such that
orienting any of them creates a cycle.

Consider the graph $H_{t-1}$ just before Maker starts round $t$.
There are ${t-1 \choose 2} - (t-1)$ edges that may close a cycle, of
them at most $(b+1)(t-1)-(t-1)$ were oriented in previous rounds. If
$(b+1)(t-1)-(t-1)<{t-1 \choose 2} - (t-1)$ at the beginning of round
$t$ then Maker wins. Unless Maker wins before that, the game lasts
at least $\frac{{n \choose 2}}{b+1}$ rounds, and therefore by taking
$t \geq \frac{{n \choose 2}}{b+1}$ we get that if $b \leq n/2-2$
then Maker surely wins. \qed

\section{The Hamiltonicity game}
\label{sec:Hamiltonicity} In this game Maker wins if the obtained
tournament contains a Hamilton cycle, and Breaker wins otherwise. We
start with the following easy and well known lemma, whose proof is
given here for completeness.
\begin{lem}
\label{lem:strong-conn-implies-hamiltonicity} Let $T$ be a strongly
connected tournament. Then $T$ contains a Hamilton cycle.
\end{lem}
\begin{proof}
Let $C = u_1,u_2,\ldots u_r,u_1$ be a longest directed cycle in $T$.
If $C$ is not a Hamilton cycle, there is a vertex $v \notin C$.
Since $T$ is strongly connected, there is a path from $v$ to $C$ and
a path from $C$ to $v$. Suppose first that $(u_i,v),(v,u_j) \in
E(T)$ for some $1 \leq i \neq j \leq r$. Without loss of generality,
assume that $j>i$. Since $T$ is a tournament, there is some index $i
\leq k \leq j-1$ such that $(u_k,v), (v,u_{k+1}) \in E(T)$ and hence
we get a  longer cycle
$u_1,u_2,\ldots,u_k,v,u_{k+1},\ldots,u_r,u_1$, a contradiction.

If there are no two indices $i,j$ such that $(u_i,v),(v,u_j) \in
E(T)$, then all the edges between $v$ and $C$ are in the same
direction. Suppose that for every $1 \leq i \leq r$, we have
$(u_i,v) \in E(T)$ (the other case is similar). Since $T$ is
strongly connected, there is a path $v,x_1,\ldots,x_t,u_i$ for some
$1 \leq i \leq r$, where the vertices $x_1\ldots,x_t$ are not in
$C$. We therefore get a longer cycle
$u_1,\ldots,u_{i-1},v,x_1,\ldots,x_t,u_i,\ldots,u_r,u_1$, a
contradiction. Therefore $T$ contains a Hamilton cycle, as claimed.
\end{proof}

We conclude that if Maker constructs a strongly connected graph from
his own edges then he wins the game.

\paragraph{Breaker's strategy.}
Assuming that the bias is sufficiently large, Breaker has a strategy
to guarantee that the obtained tournament $T$ contains a vertex with
in-degree $0$. In this case clearly $T$ does not contain a Hamilton
cycle. To this end, we reduce this problem to a {\em box game},
similarly to the treatment in~\cite{CE78}.

Let $K_n$ be the complete graph on $n$ vertices, and consider the
$b$-biased game, where $b \geq \frac{(1+o(1))n}{\ln n}$. Recall that
$H_t$ is the oriented graph obtained after $t$ rounds. Fix a
partition $V(K_n) = A \cup B$, where $A$ and $B$ are disjoint sets,
$|A| = b, |B| = n-b$. Throughout the game, Breaker orients the edges
from $A$ to $B$ until after some round $t$ there are two vertices
$u,u' \in A$ such that for every vertex $w \in B$, both
$(u,w),(u',w) \in H_t$, and the in-degree of both $u,u'$ is $0$.
Then in the last turn he orients edges within $A$ so that either $u$
or $u'$ will have in-degree $0$.

We have the following well-known result of Chv\'{a}tal and
Erd\H{o}s~\cite{CE78}.
\begin{thm}
\label{thm:Chvatal-Erdos} Suppose that there are $r$ disjoint sets
(or boxes) $B_1,\ldots,B_r$, each box $B_i$ containing $k$ elements.
At each round, Box-Maker claims $b$ elements and then Box-Breaker
claims a single element. If
$$k \leq b\sum_{i=1}^{r} \frac{1}{i},$$ then
Maker has a strategy to occupy all the elements of a single box.
\end{thm}
Note that in each round, Box-Breaker {\em destroys} a single box,
and so throughout the game Box-Maker tries to claim all elements of
a single box before it is destroyed by Box-Breaker.

Here we need a variant of this theorem, for the case that Box-Maker
actually has to complete two boxes.

\begin{clm}
\label{clm:Chvatal-Erdos-2-boxes} Suppose that there are $r$
disjoint sets, $B_1,\ldots B_r$, each $B_i$ containing $k$ elements.
At each round, Box-Breaker destroys one set and then Box-Maker
claims $b$ elements. If $$k+b \leq b\sum_{i=1}^{k} \frac{1}{i},$$
then Box-Maker has a strategy to occupy all the elements of two
boxes.
\end{clm}
\begin{proof}
For every box $B_i$ we add a set $B'_i$ of $b$ virtual items.
Consider a standard box game where the $i$'th box is $B_i \cup
B'_i$, and suppose that Box-Maker always claim the elements of $B_i$
before he claims the elements of $B'_i$, for every $1 \leq i \leq
r$. If
$$k+b \leq b\sum_{i=1}^{r} \frac{1}{i},$$ then
by Theorem~\ref{thm:Chvatal-Erdos} Box-Maker has a strategy to win
the game. Consider the last round before Box-Maker wins, when the
next move should be taken by Box-Breaker. Since Box-Breaker cannot
avoid Box-Maker's win there are at least two indices $i \neq j$ such
that all but at most $b$ elements of boxes $i$ and $j$ are already
claimed by Box-Maker. Therefore we conclude that there are at least
two indices $i \neq j$ such that $B_i$ and $B_j$ are claimed. We
conclude that Box-Maker claimed all the elements of two of the
original boxes, no matter what Box-Breaker did. The claim follows.
\end{proof}

In our setting, Maker and Breaker switch their roles. That is, we
define the boxes so that Breaker will take Box-Maker's role, and if
he claims a box the obtained tournament has a vertex of in-degree
$0$. For every vertex $v \in A$ we define a box $X_v$ as $\{vw : w
\in B\}$. Note that $|X_v|=|B|=n-b$. In every turn, Maker (that is,
Box-Breaker) can destroy one box $X_v$ by directing an edge towards
$v$, either from a vertex from $A$ or from $B$. On the other hand,
Breaker (Box-Maker) can orient $b$ edges from $A$ to $B$, which is
equivalent to taking $b$ elements from the various boxes. By
Claim~\ref{clm:Chvatal-Erdos-2-boxes}, if
\begin{equation}
\label{eqn:condition-for-box-win} n = |X_v|+b \leq b\sum_{i=1}^{|A|}
\frac{1}{i},
\end{equation}
then Breaker has a strategy to have two vertices $u,u'$ from $A$ for
which all their incident edges that connect them to $B$ are directed
towards $B$, and none of the edges from $A$ enters $u$ or $u'$.
Therefore, no matter what Maker does, Breaker can direct all the
edges from either $u$ or $u'$, thus creating a vertex with in-degree
$0$ and destroying any chance for creating a Hamilton cycle. Taking
$b \geq \frac{n(1+o(1))}{\ln n}$ satisfies
(\ref{eqn:condition-for-box-win}) and thus Breaker wins the game,
and thus Item~\ref{item:hamilton-upper-bound} in
Theorem~\ref{thm:create-hamilton} follows.

\paragraph{Maker's strategy.}

Maker's strategy consists of two stages. His goal in the first stage
is to create a graph with some expansion properties, so that all
sufficiently small sets have at least one in-going edge and at least
one out-going edge. To this end, he will create a graph with min
in-degree and out-degree at least $3$. We will show that with
positive probability (and actually, with high probability) after
this stage the graph has the desired expansion properties. Since the
game considered is a perfect information game with no chance moves,
we conclude that Maker has a \textbf{deterministic} strategy that
guarantees these properties after the first stage. Moreover, the
first stage lasts at most $8n$ rounds in any case.

At the second stage, Maker will ensure that for every large enough
disjoint sets of vertices $A,B$ there is at least one edge from $A$
to $B$ and at least one edge from $B$ to $A$. We will show that if
the he succeeds at the first stage then after the second stage we
will have a strongly connected graph and hence by the end of the
game Maker will win.

We say that a directed graph $G$ is {\em $k$-expanding} if the
following holds.
\begin{itemize}
\item For every set $A$ of size at most $k$, $|N^{+}(A)|,|N^{-}(A)|
> 0$.
\item For every two disjoint sets $A,B$ of size at least $k$, there is an
edge from $A$ to $B$ and there is an edge from $B$ to $A$.
\end{itemize}

We will show that after the first stage the obtained graph will have
the first property with high probability, and after the second stage
it will have the second property.

We have the following.
\begin{lem}
\label{lem:expanders-are-strongly-connected} Let $G$ be a directed
graph, and suppose that $G$ is $k$-expanding for some $k$. Then $G$
is strongly connected.
\end{lem}
\begin{proof}
Let $A_1,A_2,\ldots,A_t$ be the strongly connected components of
$G$, and suppose that $t>1$. Let $T$ be a graph where each $A_i$ is
represented by a vertex, and there is an edge from $A_i$ to $A_j$ if
and only if there is a vertex $v_i \in A_i$ and a vertex $v_j \in
A_j$ such that $(v_i,v_j) \in E(G)$. It is well known that $T$ is a
directed forest, and therefore contains a leaf, i.e., a set $A_i$
with no outgoing edges. If $|A_i| < k$ then since $|N^{+}(A_i)| > 0$
we get a contradiction. If $|A_i| > n-k$, then since $|N^-(V
\setminus A_i)| > 0$ we get a contradiction. Finally, if $k \leq
|A_i| \leq n-k$, then by the second property there is an edge from
$A_i$ to $V \setminus A_i$. Therefore, we conclude that $t=1$ and
hence $G$ is strongly connected.
\end{proof}

More specifically, we will show that for $k
=\frac{n}{(\ln{n})^{2/5}}$, at the first stage Maker ensures that
for every set $A$ of size at least $k$, $|N^{+}(A)|, |N^-(A)|>0$,
and at the second stage Maker ensures that for every two sets $A,B$
of size at least $k$, there is an edge from $A$ to $B$. By
Lemma~\ref{lem:expanders-are-strongly-connected} and
Lemma~\ref{lem:strong-conn-implies-hamiltonicity}, after the second
stage Maker wins.

\paragraph{The first stage.} At the first stage we adapt the techniques of Gebauer and
Szab\'{o}~\cite{GS09} in a way similar to~\cite{Krivelevich11} and
show that if $b = \frac{(1-o(1))n}{\ln n}$ then Maker has a winning
strategy. We start by reducing our game to an undirected game on the
edges of a bipartite graph.

Suppose that Maker and Breaker play a biased orientation game on the
edges of the complete graph $G=(V,E)$ on $n$ vertices, and let
$V=\{v_1,v_2,\ldots ,v_n\}$. Let $H=(V_1,V_2,E')$ be the complete
bipartite graph on $2n$ vertices, where $V_1 = \{v_{1,1}, v_{1,2},
\ldots ,v_{1,n} \}$ and $V_2 = \{v_{2,1}, v_{2,2}, \ldots, v_{2,n}
\}$. Throughout the game we maintain two subgraphs, $H_M$ consisting
of edges that are associated with Maker and $H_B$ consisting of
edges that are associated with Breaker. Initially both graphs are
empty.

If Breaker orients a previously undirected edge from $v_i$ to $v_j$
in $G$, we add the edge between $v_{2,i}$ and $v_{1,j}$ to $H_B$.

Maker, in his turn, would like to create a graph with a constant
minimum degree in $H_M$. Whenever Maker, according to the strategy
to be described below, wants to add some edge $(v_{1,i},v_{2,j})$ to
$H_M$ and $(v_{2,i},v_{1,j})$ has not been taken yet, he does it and
also orients $v_i$ to $v_j$ in $G$ (note that in this case the edge
between $v_i$ and $v_j$ is undirected before this step). In this
case we also add the edge $(v_{2,i},v_{1,j})$ to $H_B$. If, on the
other hand, $(v_{2,i},v_{1,j}) \in E(H_B)$, then he adds
$(v_{1,i},v_{2,j})$ to $H_M$, and then plays another turn by taking
a free edge according to his strategy, and adding the opposite edge
to Breaker's graph. Finally, if Maker takes an edge
$(v_{1,i},v_{2,i})$ then he plays another turn. Since the classical
Maker-Breaker game is bias-monotone, if Maker takes more than one
edge it can only help him. Also, note that edges from $v_{1,i}$ to
$v_{2,i}$ are useless for Maker in the real game. Therefore Maker
will have to construct a graph with minimum degree $c+1$ so that
every vertex has at least $c$ neighbors other than himself.

Observe also that $(v_{2,i},v_{1,j}) \notin E(H_M)$, as otherwise
$(v_{1,i},v_{2,j})$ would be added to $H_B$ in some previous step.
The following proposition summarizes this reduction.

\begin{prop}
\label{prop:reduction-to-bipartite-graphs} If at some step $H_M$ has
minimum degree $c+1$ then at the same time every vertex in $G$ has
min in-degree and out-degree at least $c$.
\end{prop}

\paragraph{Gebauer-Szab\'{o} proof.} In~\cite{GS09}, Gebauer and
Szab\'{o} provided a strategy for Maker (in the classical
Maker-Breaker setting) to construct a spanning tree, a graph with
postive minimum degree, and a connected graph with high minimum
degree when $b = \frac{(1+o(1))n}{\ln n}$. Here we summarize their
method and highlight the slight differences between their strategy
for the min-degree game and what we need in our case. We refer the
reader to~\cite{GS09} for a complete proof. Their strategy is
defined as follows. The goal of Maker is to construct a graph with
min-degree $c$. Throughout the game, a vertex $v$ is {\em dangerous}
if $d_M(v) \leq c-1$. Define the danger value of $v$ as $\danger(v)
= d_B(v)-2b\cdot d_M(v)$. Initially, the danger of all vertices is
$0$. At every round, Maker takes a vertex $v$ with maximum danger
value (ties are broken arbitrarily), and then takes an arbitrary
unclaimed edge incident to $v$. The proof goes by assuming a
Breaker's win, and analyzing the change of danger value of the
vertices for which Maker took incident edges in the game, and
showing that the average danger value must be greater than $0$. This
in turn would lead to a contradiction.

In our case, our board consists of the edges of the complete
bipartite graph $K_{n,n}$ instead of the edges of the complete graph
$K_n$. Moreover, when Maker claims an edge, Breaker may get the
opposite edge as well; We add this edge to the next move of Breaker.
Therefore, Maker plays against Breaker that claims at most $(b+1)$
edges in his turn. The danger of a vertex is defined only with
respect to edges (and degrees) that belong to the bipartite graph
$K_{n,n}$, and hence at the beginning of the game the danger of
every vertex is $0$. The rest of the analysis is essentially the
same as~\cite{GS09}.

It was observed in~\cite{Krivelevich11} that Maker can achieve the
minimum degree $c$ at every vertex before Breaker claimed
$(1-\delta)n$ of its incident edges, for $\delta =
\frac{15}{(\ln{n})^{1/4}}$. Also, if Maker claims an edge that is
incident to a vertex $v$, he chooses one of the edges randomly and
uniformly among the free incident edges. Note this in this case
Breaker also gets only one new edge.

We will show that after the first stage, the obtained graph has
typically some expanding properties. In our case after at most $8n$
rounds, $H_M$ has min-degree at least $4$, which results in an
oriented graph with the property that every vertex has in-degree and
out-degree at least $3$. Observe that this stage lasts at most $8n$
moves as in every round Maker increases the degree of one of the
vertices in $K_{n,n}$ by at least one.

We conclude the description of this approach with the following
proposition.
\begin{prop}
\label{prop:Gebauer-Szabo-app} Suppose that $b = \frac{(1-o(1))\cdot
n}{\ln{n}}$. Then Maker has a strategy to construct after at most
$8n$ turns a directed graph with min in-degree and min out-degree at
least $4$. Moreover, throughout the game, Maker chooses at each turn
a vertex $v$ according to his strategy, and picks a random incident
edge out of a set of at least $\delta n$ choices, where $\delta =
\frac{15}{(\ln n)^{1/4}}$.
\end{prop}

For completeness we provide the proof details in the appendix.

\paragraph{Applying Gebauer-Szab\'{o} approach.}

Let $A$ be a set of vertices of size $O(\frac{n}{(\ln{n})^{2/5}})$.
We next prove that almost surely after the first stage $A$ has at
least one ingoing edge and at least one outgoing edge. We start by
claiming that almost surely every such set has at least one ingoing
edge. Observe first that the property trivially holds for every set
with a single vertex, as every vertex has in-degree at least one.
Consider a fixed set $A$ of size $i$, and assume that $A$ has no
ingoing edges, then all edges that enter $A$ have their other
endpoint also in $A$, and there are at least $3i$ such edges. By
Proposition~\ref{prop:Gebauer-Szabo-app}, whenever Maker chooses a
dangerous vertex $v$ from $A$, there are at least $\delta n$
unclaimed edges incident to $v$. Therefore, the probability that
Maker chooses an edge between $v$ and another vertex of $A$ is at
most $\frac{|A|-1}{\delta n -1}$. After the first stage there are
$3i$ ingoing edges to vertices of $A$, hence the probability that
$A$ does not have even a single ingoing edge from a vertex outside
$A$ is at most $(\frac{|A|-1}{\delta n -1})^{3i}$. Therefore, by the
union bound, the probability that there is set $A$ of size $i$ with
no ingoing edge is at most
$$
{n \choose i} \cdot \left(\frac{|A|-1}{\delta n -1}\right)^{3i} \leq
\left(\frac{en}{i}\right)^i \cdot \left(\frac{2i}{\delta
n}\right)^{3i} \leq \left(\frac{8ei^2}{\delta^3 n^2}\right)^i.
$$

By considering the two cases when $i \leq n^{1/3}$ and $i \geq
n^{1/3}$ it is easy to check that for every $2 \leq i \leq \frac{
n}{(\ln{n})^{2/5}}$ and $\delta = \frac{15}{(\ln{n})^{1/4}}$, the
last expression is bounded by $o(1/n)$. Therefore by the union bound
every set of size at most $\frac{n}{(\ln{n})^{2/5}}$ has at least
one ingoing edge, assuming that $n$ is sufficiently large.
Essentially the same argument shows that almost surely every such
set of that size contains at least one outgoing edge, as claimed.

Clearly, the first stage takes at most $8n$ rounds, so the total
number of taken edges is at most $8n(b+1)$.

\paragraph{The second stage.}

Recall that at the second stage Maker has to connect in both
directions every two disjoint sets $A,B$ of size $\frac{(1-o(1))
n}{(\ln{n})^{2/5}}$.

Consider a random tournament obtained from $K_n$ by directing each
edge uniformly and independently of the other edges. For every two
disjoint sets of vertices $A,B$, the number of edges from $A$ to $B$
is binomially distributed. Denote by $e(A,B)$ the number of edges
from $A$ to $B$ and by $e(B,A)$ the number of edges from $B$ to $A$.
By the Chernoff bound (see, e.g.,~\cite{AS08}), we have
$$ Pr\left[\left|e(A,B)-\frac{|A||B|}{2}\right| \geq \eps|A||B|\right] \leq e^{-\eps^2
|A||B|/2},$$ and similar inequality holds also for $e(B,A)$.

Therefore, if $|A|=|B|=\frac{(1-o(1)) n}{(\ln{n})^{2/5}}$ the
probability that $e(A,B)$ or $e(B,A)$ is greater than $\tfrac{1}{2}
\cdot |A||B|(1+n^{-1/2+o(1)})$ is $2^{-2n}$, and hence by the union
bound for every two such sets, there are at least $\tfrac{1}{2}
\cdot |A||B|(1+n^{-1/2+o(1)}) \geq \frac{0.99 n^2}{2(\ln{n})^{4/5}}$
in each direction. Fix a tournament $T^*$ with this property.

At the second stage, Maker always directs edges that agree with
$T^*$. That is, he can only direct an edge from $u$ to $v$ if $(u,v)
\in E(T^*)$. For every two such sets, at most $8n(b+1) \leq
\frac{12n^2}{\ln{n}}$ edges were directed at the first stage of the
game, and hence at the beginning of the second stage at least
$\frac{0.99 n^2}{2(\ln{n})^{4/5}}$ edges that are directed from $A$
to $B$ in $T^{*}$ are unclaimed.

Now Maker and Breaker switch roles. Maker clearly wins if he
prevents Breaker from claiming all the edges from a set $A$ to a set
$B$, where $|A|=|B| = k = \frac{n}{(\ln{n})^{2/5}}$. To this end, we
apply the Beck-Erd\H{o}s-Selfridge criteria
(Theorem~\ref{thm:Erdos-selfride-beck}), with
$p=b=\frac{n(1+o(1))}{\ln{n}}$, $q=1$, the size of each hyperedge is
at least $\frac{0.99 n^2}{2(\ln{n})^{4/5}}$ and the total number of
sets is at most ${n \choose k}^2$. We have

\begin{align*}
\sum_{A \in \mathcal F} (q+1)^{-\frac{|A|}{p}} & <   {n \choose k}^2 \cdot (q+1)^{-\frac{|A|}{p}}  \\
& < \left(\frac{en}{k}\right)^{2k} \cdot 2^{-\frac{n \ln{n}}{3(\ln{n})^{4/5}}}  \\
& \leq 2^{\frac{4n \cdot \log{\log{n}}}{(\ln{n})^{2/5}}} \cdot
2^{-\frac{n (\log{n})^{1/5}}{6}} \ll 1.
\end{align*}

Therefore in our case Maker wins and hence every two sets of size
$k$ are connected in both ways.

We conclude that at the end of the second stage Maker has a strongly
connected graph, and hence by the end of the game the obtained
tournament is strongly connected, and Maker wins. This proves
Item~\ref{item:hamilton-lower-bound} in
Theorem~\ref{thm:create-hamilton}. \qed

\section{The $H$-creation game}
\label{sec:h-creation}

In this game, a fixed oriented graph $H$ is given. Maker wins if the
obtained tournament contains a copy of $H$, and Breaker wins
otherwise. Note that if $H$ does not contain a directed cycle then
Maker surely wins for large enough $n$, as every tournament of size
$n$ contains a transitive tournament of size $\log{n}$, which
contains $H$ as a subgraph.

Our starting point is an upper bound on the bias threshold for a
general fixed graph $H$. Given a directed graph $H$, and a bijection
$\sigma : V(H) \to |V(H)|$, we define the {\em feedback arc set} of
$H$ with respect to $\sigma$ as
$$\FAS(H,\sigma) = |\{(u,v) \in E(H) : \sigma(u)>\sigma(v)\}|.$$
In words, this parameter measures the number of edges that are going
in the wrong direction with respect to $\sigma$. Let
$$
\FAS(H) = \min_{\sigma} \{\FAS(H,\sigma)\}.
$$
This is the minimal number of edges of $H$ that has to be deleted in
order to make $H$ an acyclic graph. If, for example, $H$ is a random
tournament on $t$ vertices, then it is easy to show that with high
probability $\FAS(H)$ is close to $t(t-1)/4$.

We have the following upper bound.
\begin{lem}
Let $H$ be a graph on $t$ vertices, and let $r = \FAS(H)$. Suppose
that Maker and Breaker play an orientation game on $K_n$. Then if
$b>c(H) \cdot n^{t/r}$ then Breaker has a strategy guaranteeing that
the obtained tournament does not contain a copy of $H$, where
$c(H)>0$ depends only on $H$.
\label{lem:H-creating-upper-bound-with-FAS}
\end{lem}
\begin{proof}
The proof follows by a simple application of the
Beck-Erd\H{o}s-Selfridge theorem
(Theorem~\ref{thm:Erdos-selfride-beck}). Breaker will choose an
arbitrary bijection $\sigma$ of the vertices, and at every turn he
will direct the edges according to $\sigma$. That is, whenever he
chooses to direct an edge $uv$, and $\sigma(v)>\sigma(u)$ the edge
will be directed from $u$ to $v$. Hence, if Maker creates a copy of
$H$, by definition he orients at least $\FAS(H)$ edges in the
opposite direction with respect to $\sigma$. We can thus reduce the
game to the classical Maker-Breaker game as follows. In every set of
$t$ vertices, Maker can win only if he claims at least $r$ edges
that are induced by this set, and Breaker wins if he prevents Maker
from doing so. The total number of winning sets for Maker is at most
${n \choose t} \cdot {{t \choose 2} \choose r}$. Therefore, if
$(q+1)^{r} = \Omega({n \choose t} \cdot {{t \choose 2} \choose r})$
then by Theorem~\ref{thm:Erdos-selfride-beck}, Breaker has a winning
strategy. This is the case if $b >  c(t,r) \cdot n^{t/r}$, and hence
the lemma holds.
\end{proof}

It is worth noting that following the methods of Bednarska and
{\L}uczak~\cite{BL00}, one can prove that if $b =
O(n^{\frac{|V(H)|-2}{|E(H)|-1}})$ then Maker has a winning strategy
as follows. Maker chooses at each round a random undirected edge and
orients it randomly, independently of the other choices. Roughly
speaking, one can show that by the end of the game the obtained
graph looks random in some sense, and hence if the bias is large
enough then with high probability it contains a copy of $H$.

However, following their approach does not give sharp bounds in our
case. To see this, observe that their results give a sharp bound of
$b = \Theta(\sqrt{n})$ for the triangle creation game in the
classical Maker-Breaker settings, while in orientation games the
correct bias for creating a cyclic triangle or even any directed
cycle is $b = \Theta(n)$, as we will see shortly.

We next generalize the result of Section~\ref{section:cycle-game}
and show that in the case that $H$ is a fixed cycle, Maker wins even
if $b = \Omega(n)$.
\begin{prop}
\label{prop:maker-creates-cycle} For every constant $k\geq 3$ there
is a constant $\gamma(k)>0$ such that if $b < \gamma(k) \cdot n$
then Maker wins the $b$-biased $C_{k}$-creation game.
\end{prop}
\begin{proof}
We first observe that if a tournament $T$ contains a cycle of length
$k+(k-2)r$ for some $r \in \N$ then it also contains a cycle of
length $k$. The proof of this observation is by induction. It is
trivially true for $r=0$. Suppose this is true for all values
smaller than some fixed $r$, and let
$v_1,v_2,\ldots,v_{k+(k-2)r},v_1$ a cycle of length $k+(k-2)r$.
Consider the edge between $v_k$ and $v_1$. If the edge is directed
from $v_k$ to $v_1$ there is a cycle of length $k$ and we are done.
Otherwise, $v_k,v_{k+1},\ldots,v_{k+(k-2)r},v_1,v_k$ is a cycle of
length $k+(k-2)(r-1)$ and therefore by the induction hypothesis $T$
contains a cycle of length $k$, as required.

Therefore, in order to create a cycle $C_k$, Maker has to create
some cycle of length $k+(k-2)r$. By
Lemma~\ref{lem:maker-strategy-creating-cycle}, at each round Maker
can extend a longest directed path by $1$. Maker's strategy is to
close a cycle of length $k+(k-2)r$ whenever it is possible, and to
extend a longest directed path by $1$ if it is not possible. After
$t$ rounds, there is a path of length $t$, and we denote it by
$x_1,\ldots,x_t$. For every $i \geq k$, the number of edges from
$x_i$ to $x_j$, $j<i$, that may close a cycle of length $k+(k-2)r$
is at least $\frac{i}{k-2}-2$. Hence the total number of edges that
may close a cycle of length $k+(k-2)r$ for some $r$ is at least
$$\sum_{i=k}^{t} \left( \frac{i}{k-2}-2 \right) \geq
\frac{(t+k)(t-k)}{2(k-2)}-2t.$$ Therefore, the number of such edges
is at least $\frac{{t \choose 2}}{k}$ for $t = \Omega(k^2)$, that is
at least $(1/k)$-fraction of the edges for such $t$. Among these
$\frac{{t \choose 2}}{k}$ edges, at most $(b+1)t$ were oriented by
either Maker or Breaker in previous rounds. Note that as long as the
bias $b$ is smaller than $n/2$, the game lasts at least $t = n$
rounds, and results in a path of length $n-1$, unless Maker wins
before. Therefore if $b<\frac{n-1-2k}{2k}$ then Maker wins the game,
as required.
\end{proof}

Recall that a tournament $T$ is $k$-colorable if its edges can be
partitioned into $k$ transitive tournaments. Berger et
al.~\cite{BCCFLSST11} studied the class of tournaments $H$ with the
property that there is constant $c(H)$ such that every $H$-free
tournament $T$ is $c(H)$-colorable. They called every such
tournament $H$ a {\em hero}, and characterized the set of such
tournaments. We next show that for every $k>0$ Maker has a strategy
to create a non $k$-colorable tournament as long as the bias is a
sufficiently small linear function of $n$.

\begin{lem}
\label{lem:maker-strategy-creating-tournament-with-high-chromatic-number}
Let $k>0$, and suppose that $b = \frac{cn}{k \log{k}}$ for some
sufficiently small constant $c>0$. Then Maker has a strategy to
create a non $k$-colorable tournament.
\end{lem}
\begin{proof}
It is rather easy to see using a Chernoff bound (as it was done in
Section~\ref{sec:Hamiltonicity}) that  a random tournament obtained
by directing each edge uniformly and independently of the other
choices has typically the following property. For every ordered pair
of disjoint sets $A,B$ of size $n/2k$, there are
$\Theta(\frac{n^2}{k^2})$ edges in each direction between $A$ and
$B$. Fix a tournament $T^*$ with this property.

Define a hypergraph $H$ whose vertices are the edges of $T^*$ and
whose edges are all the edges from $A$ to $B$ in $T^*$ for every
ordered pair $A,B$ of size $n/2k$. Maker will win the game by
orienting one edge from every hyperedge in $H$ according to $T^*$.
To this end, Maker will play to prevent Breaker from orienting all
the edges in some hyperedge from $H$. By the end of the game, there
is an edge between every two sets of size $n/2k$ and hence the
obtained tournament does not contain an acyclic set of size $n/k$,
and therefore is not $k$-colorable.

There are ${n \choose n/2k}^2 \leq (2ek)^{n/k}$ choices of ordered
pairs $(A,B)$, each corresponding to a hyperedge of $H$. The size of
each hyperedge is $\Theta(\frac{n^2}{k^2})$. By applying the
Beck-Erd\H{o}s-Selfridge strategy
(Theorem~\ref{thm:Erdos-selfride-beck}, with Maker playing role of
Breaker, $p=b$ and $q=1$), if $b=\frac{cn}{k\log{k}}$ then Maker has
a winning strategy, as required.
\end{proof}

A simple consequence of
Lemma~\ref{lem:maker-strategy-creating-tournament-with-high-chromatic-number}
is the following generalization of
Lemma~\ref{lem:maker-strategy-creating-cycle}. Berger et
al.~\cite{BCCFLSST11} provided a list of five minimal tournaments
$H_1,H_2,\ldots H_5$, and proved (Theorem 5.1 in~\cite{BCCFLSST11})
that every non-hero tournament must contain at least one of
$H_1,\ldots, H_5$ as a subtournament. For every $1 \leq i \leq 5$,
one can check that $\FAS(H_i) \geq 2$.

Consider any oriented graph $H$ with $\FAS(H)=1$. Let $\sigma$ be an
ordering of $V(H)$ with a single edge that does not agree with
$\sigma$. Let $H'$ be a tournament on $V(H)$ that contains $H$ as
subgraph and is defined as follows. For every two vertices $u,v \in
V(H)$, if $(u,v) \in E(H)$ we let $(u,v) \in E(H')$. If $(u,v),(v,u)
\notin E(H)$, we let $(u,v) \in E(H')$ if $\sigma(v)>\sigma(u)$ and
$(v,u) \in E(H')$ otherwise. We get that $\FAS(H') = 1$ as well.

Clearly, it is sufficient to construct a copy of $H'$ for Maker's
win. The result of Berger et al.~\cite{BCCFLSST11} can be applied
only for tournaments, and hence we will use it to show that Maker
can construct a copy of $H'$.

Since $\FAS(H')=1$ then $H'$ is a hero, and therefore every
tournament that does not contain a copy of $H'$ is
$c(H')$-colorable, where $c(H')$ is a constant that depends only on
$H'$. By
Lemma~\ref{lem:maker-strategy-creating-tournament-with-high-chromatic-number},
if $b=\Theta(n)$ then Maker has a strategy so the obtained
tournament is not $c(H')$-colorable. We therefore have the
following.

\begin{prop}
\label{prop:Maker-wins-graphs-with-single-reserved-edge} For every
oriented graph $H$ with $\FAS(H)=1$ there is a constant
$\gamma(H)>0$ such that Maker wins the $\gamma n$-biased $H$
creation game.
\end{prop}

We conjecture that the bias threshold that guarantees Maker's win
strongly depends on $\FAS(H)$. It will be interesting to find
further quantitative results in this direction.

\appendix

\section{Proof of Proposition~\ref{prop:Gebauer-Szabo-app}}

Here we provide the complete details of Gebauer-Szab\'{o} approach
and show that Maker can win the min-degree game if $b =
\frac{(1-o(1))n}{\ln n}$ where the base graph is the complete
bipartite graph $K_{n,n}$. We give the proof for every min-degree
$c$, though we need only the case $c=4$.

We assume for simplicity that Breaker starts the game, this does not
change the asymptotic threshold of this game. We say that the game
ends when either all vertices have degree at least $c$ in Maker's
graph (and Maker won) or one vertex has degree at least $n - c + 1$
in Breaker's graph (and Breaker won). With $\degm(v)$ and $\degb(v)$
we denote the degree of a vertex $v$ in Maker's graph and in
Breaker's graph, respectively. A vertex $v$ is called
\emph{dangerous} if $\degm(v) \leq c  -1 $. To establish Maker's
strategy we define the \emph{danger value} of a vertex $v$ as
$\danger(v) := \degb(v) - 2b \cdot \degm(v)$.

\paragraph{Maker's strategy $S_M$}
Before his $i$th move Maker identifies a dangerous vertex $v_i$ with
the largest danger value, ties are broken arbitrarily. Then, as his
$i$th move Maker claims an edge incident to $v_i$. We refer to this
step as ``easing $v_{i}$''.

Observe that Maker can always make a move according to his strategy
unless no vertex is dangerous (thus he won) or Breaker occupied at
least $n-c+1$ edges incident to a vertex (and Breaker won).

Also, a vertex $v_i$ was dangerous any time before Maker's $i$th
move.

Suppose, for a contradiction, that Breaker, playing with bias $b$,
has a strategy $S_B$ to win the min-degree-$c$ game against Maker
who plays with bias $1$. Let $B_{i}$ and $M_{i}$ denote the $i$th
move of Breaker and Maker, respectively, in the game where they play
against each other using their respective strategies $S_B$ and
$S_M$.  Let $g$ be the length of this game, i.e., the maximum degree
of Breaker's graph becomes larger than $n-c$ in move $B_g$. We call
this the end of the game.

For a set $I\subseteq V$ of vertices we let $\avdanan( I )$ denote
the average danger value $\frac{\sum_{v\in I} \danger(v)}{|I|}$ of
the vertices of $I$. When there is risk of confusion we add an index
and write $\danger_{B_{i}}(v)$ or $\danger_{M_{i}}(v)$ to emphasize
that we mean the danger-value of $v$ \emph{directly before} $B_{i}$
or $M_{i}$, respectively.

In his last move Breaker takes $b$ edges to increase the maximum
Breaker-degree of his graph to at least $n-c$ (in fact, at least
$n-c+1$). In order to be able to do that, directly before Breaker's
last move $B_g$ there must be a dangerous vertex $v_g$ whose
Breaker-degree is at least $n-c-b$. Thus $\danger_{B_g}(v_g) \geq
n-c-b - 2b(c-1).$

Recall that $v_1, \ldots , v_{g-1}$ were defined during the game.
For $0\leq i \leq g-1$, we define the set $I_i$ as $I_i=\{ v_{g-i},
\ldots v_g\}$.

The following lemma estimates the change in the average danger
during Maker's move.
\begin{lem} \label{lem:makermindeg}
Let $i$, $1\leq i \leq g-1$,

$(i)$ if $I_i\neq I_{i-1}$, then $\avdanan_{M_{g-i}}(I_{i}) -
\avdanan_{B_{g-i+1}}(I_{i-1}) \geq 0.$

$(ii)$ if $I_i=I_{i-1}$, then $\avdanan_{M_{g-i}}(I_{i}) -
\avdanan_{B_{g-i+1}}(I_{i-1}) \geq \frac{2b}{|I_i|}.$
\end{lem}
\emph{Proof:} For part $(i)$, we have that $v_{g-i}\notin I_{i-1}$.
Since danger values do not increase during Maker's move we have
$\avdanan_{M_{g - i}}(I_{i-1}) \geq \avdanan_{B_{g - i +
1}}(I_{i-1})$. Before $M_{g-i}$ Maker selected to ease vertex
$v_{g-i}$ because its danger was highest among dangerous vertices.
Since all vertices of $I_{i-1}$ are dangerous before $M_{g-i}$ we
have that $\danger(v_{g-i}) \geq \max(\danger(v_{g-i+1}), \ldots,
\danger(v_{g}))$, which implies $\avdanan_{M_{g - i}}(I_i) \geq
\avdanan_{M_{g - i}}(I_{i-1})$. Combining the two inequalities
establishes part $(i)$.

For part $(ii)$, we have that $v_{g-i}\in I_{i-1}$. In $M_{g - i}$
$\degm(v_{g-i})$ increases by 1 and $\degm(v)$ does not decrease for
any other $v\in I_{i}$. Besides, the degrees in Breaker's graph do
not change during Maker's move. So $\danger(v_{g-i})$ decreases by
$2b$, whereas $\danger(v)$ do not increase for any other vertex
$v\in I_i$. Hence $\avdanan(I_i)$ decreases by at least
$\frac{2b}{|I_i|}$, which implies $(ii)$. \hfill $\Box$

The next lemma bounds the change of the danger value during
Breaker's moves.
\begin{lem}\label{lem:breakermindeg}
Let $i$ be an integer, $1\leq i\leq g-1$.

$(i)$ $\avdanan_{M_{g-i}}(I_i) - \avdanan_{B_{g-i}}(I_i) \leq
\frac{2b}{|I_i|}$

$(ii)$ $\avdanan_{M_{g-i}}(I_i) - \avdanan_{B_{g-i}}(I_i) \leq
\frac{b+|I_i|-1 +a(i-1)-a(i)}{|I_i|}$, where $a(i)$ denotes the
number of edges spanned by $I_i$ which Breaker took in the first
$g-i-1$ rounds.
\end{lem}
\emph{Proof:} Let $e_{\text{double}}$ denote the number of those
edges with both endpoints in $I_i$ which are occupied by Breaker in
$B_{g - i}$. Then the increase of $\sum_{v \in I_{i}} \deg_B(v)$
during $B_{g -i}$ is at most $b+\ed$. Since the degrees in Maker's
graph do not change during Breaker's move the increase of
$\avdanan(I_i)$  (during  $B_{g - i}$) is at most
$\frac{b+\ed}{|I_i|}$.

Part $(i)$ is then immediate after noting that $\ed\leq b$.

For $(ii)$,  we bound $e_{\text{double}}$ more carefully. By
definition, Breaker occupied $a(i)$ edges spanned by $I_i$ in his
first $g-i-1$ moves.  So, all in all, Breaker occupied
$a(i)+e_{\text{double}}$ edges spanned by $I_i$ in his first $g-i$
moves. On the other hand, we know that among these edges exactly
$a(i-1)$ are spanned by $I_{i-1} \supseteq I_i\setminus \{v_{g-i}\}$
and there are at most $|I_i|-1$ edges in $I_i$ incident to
$v_{g-i}$. Hence $a(i)+e_{\text{double}}\leq a(i-1) + |I_i|-1$,
giving us $ e_{\text{double}} \leq |I_i| -1 + a(i - 1) - a(i).$
\hfill $\Box$

The following estimates for the change of average danger during one
full round are immediate corollaries of the previous two lemmas.

\begin{corollary}\label{coro:danger-change}
Let $i$ be an integer, $1\leq i\leq g-1$.

$(i)$ if $I_i=I_{i-1}$, then $\avdanan_{B_{g-i}}(I_{i}) -
\avdanan_{B_{g-i+1}}(I_{i-1}) \geq 0.$

$(ii)$ if $I_i\neq I_{i-1}$, then $\avdanan_{B_{g-i}}(I_{i}) -
\avdanan_{B_{g-i+1}}(I_{i-1}) \geq -\frac{2b}{|I_i|}$

$(iii)$ if $I_i\neq I_{i-1}$, then $\avdanan_{B_{g-i}}(I_{i}) -
\avdanan_{B_{g-i+1}}(I_{i-1}) \geq
-\frac{b+|I_i|-1+a(i-1)-a(i)}{|I_i|}$, where $a(i)$ denotes the
number of edges spanned by $I_i$ which Breaker took in the first
$g-i-1$ rounds.
\end{corollary}

Using Corollary~\ref{coro:danger-change} we derive that before
$B_1$, $\avdanan(I_{g-1}) > 0$, which contradicts the fact that at
the beginning of the game every vertex has danger value $0$.

Let $k:= \lfloor \frac{n}{\ln n} \rfloor$. For the analysis, we
split the game into two parts: The main game, and the end game which
starts when $|I_{i}| \leq k$.

Let $|I_g|=r$. Let $i_1 < \ldots < i_{r-1}$ be those indices for
which $I_{i_j}\neq I_{i_j-1}$. Note that $|I_{i_j}|=j+1$. Observe
that by definition $a(i_{j-1})\geq a(i_{j}-1)$.

Recall that the danger value of $v_g$ directly before $B_g$ is at
least $n-c-b(2c-1)$.

Assume first that $k> r$.
\begin{eqnarray}
\avdanan_{B_1}(I_{g-1}) & = & \avdanan_{B_g}(I_{0}) +
\sum_{i=1}^{g-1}
\left( \avdanan_{B_{g-i}}(I_{i}) - \avdanan_{B_{g-i+1}}(I_{i-1}) \right) \nonumber \\
& \geq & \avdanan_{B_g}(I_{0}) + \sum_{j=1}^{r-1} \left(
\avdanan_{B_{g-i_j}}(I_{i_j}) - \avdanan_{B_{g-i_j+1}}(I_{i_j-1})
\right)
\enspace\enspace \mbox{[by Corollary~\ref{coro:danger-change}$(i)$]} \nonumber \\
& \geq & \avdanan_{B_g}(I_{0}) - \sum_{j=1}^{r-1}  \frac{b + j +
a(i_j-1) - a(i_j)}{j+1}
\enspace\enspace \mbox{[by Corollary~\ref{coro:danger-change}$(iii)$]} \nonumber \\
& \geq &  \avdanan_{B_g}(I_{0}) - b H_r - r - \frac{a(0)}{2} +
\sum_{j=2}^{r - 1}  \frac{a(i_{j-1})}{(j+1)j} + \frac{a(i_{r-1})}{r}
\enspace [\mbox{since $a(i_{j-1})\geq a(i_{j}-1)$}] \nonumber \\
& \geq &  \avdanan_{B_g}(I_{0}) - b H_k- k
\enspace \enspace \enspace  \text{[since $a(0)=0$ and $r\leq k$]} \nonumber \\
& \geq &  n-c- b (2c+ \ln k ) - k \nonumber \\
& \geq &  n-\frac{n}{\ln n} (2c+ \ln n - \ln\ln n)- \frac{n}{\ln n } -c  \enspace\enspace\enspace \mbox{[since $b \leq \frac{n}{\ln n}$]}\nonumber \\
& \geq &  \frac{n\ln\ln n}{3\ln n} - \frac{n}{\ln n} -c \nonumber \\
& > & 0. \enspace\enspace\enspace \mbox{[for large $n$]}
\label{eq:lastlinehugecalc}
\end{eqnarray}
Assume now that $k\leq r$.
\begin{eqnarray}
\avdanan_{B_1}(I_{g-1}) & = & \avdanan_{B_g}(I_{0}) +
\sum_{i=1}^{g-1}
\left( \avdanan_{B_{g-i}}(I_{i}) - \avdanan_{B_{g-i+1}}(I_{i-1}) \right) \nonumber \\
& \geq & \avdanan_{B_g}(I_{0}) + \sum_{j=1}^{r-1} \left(
\avdanan_{B_{g-i_j}}(I_{i_j}) - \avdanan_{B_{g-i_j+1}}(I_{i_j-1})
\right)
\enspace\enspace \mbox{[by Corollary~\ref{coro:danger-change}$(i)$]} \nonumber \\
& = & \avdanan_{B_g}(I_{0}) + \sum_{j=1}^{k - 1} \left(
\avdanan_{B_{g-i_j}}(I_{i_j}) - \avdanan_{B_{g-i_j+1}}(I_{i_j-1})
\right)
+ \nonumber \\
&  & \sum_{j=k}^{r - 1}
\left( \avdanan_{B_{g-i_j}}(I_{i_j}) - \avdanan_{B_{g-i_j+1}}(I_{i_j-1}) \right) \nonumber \\
& \geq & \avdanan_{B_g}(I_{0}) - \sum_{j=1}^{k - 1}  \frac{b + j +
a(i_j-1) - a(i_j)}{j+1}- \sum_{j=k}^{r - 1}  \frac{2b}{j+1}
\enspace\enspace \mbox{[by Corollary~\ref{coro:danger-change}$(iii)$ and $(ii)$]} \nonumber \\
& \geq &  \avdanan_{B_g}(I_{0}) - b (2H_r-H_k) - k - \frac{a(0)}{2} + \sum_{j=2}^{k - 1}  \frac{a(i_{j-1})}{(j+1)j} + \frac{a(i_{k-1})}{k} \nonumber \\
& \geq &  n-c- b (2c-1+2H_{2n}-H_k) - k
\enspace \enspace \enspace [\mbox{since $2n\geq r$ and $a(0)=0$}] \nonumber \\
& \geq &  n - c -  \left(\frac{n}{\ln n} - \frac{n \ln \ln
n}{\ln^2n} - (2c+3)\frac{n}{\ln^2 n} \right) (\ln n + \ln\ln n + 2c
+ 2)
- \frac{n}{\ln n} \nonumber \\
& \geq &  \frac{n(\ln\ln n)^2}{\ln^2 n} \enspace \enspace\enspace
[\mbox{for $n$ large enough}]
\nonumber \\
& > & 0. \label{eq:lastlinehugecalcsecond}
\end{eqnarray}

Observe that in our proof we need Maker to have min-degree $c$ for
every vertex $v$ before Breaker claims $(1-\delta)n$ edges incident
to $v$ (for $\delta = O(\frac{1}{(\ln{n})^{1/4}}))$. The same
analysis essentially holds, with the following differences. Assume
that Breaker wins, then before his last move the vertex $v$ has
degree $(1-\delta)n-c-1$ (instead of $n-c-1$). All other
calculations are essentially the same by taking $b = \frac{n}{\ln
n}(1-\frac{1}{\ln\ln n})$.

\end{document}